\documentclass[final]{amsart}
\usepackage{amssymb}
\usepackage{array}
\usepackage{amsthm}
\usepackage[all]{xy}

\newtheorem{theorem}{Theorem}[section]
\newtheorem{lemma}[theorem]{Lemma}
\newtheorem{proposition}[theorem]{Proposition}
\newtheorem{corollary}[theorem]{Corollary}
\newtheorem{claim}[theorem]{Claim}

\theoremstyle{definition}

\newtheorem{remark}[theorem]{Remark}
\newtheorem{question}[theorem]{Question}

\newcommand{\IN}{\mathbb N}
\newcommand{\e}{\varepsilon}
\newcommand{\IZ}{\mathbb Z}
\newcommand{\IC}{\mathbb C}
\newcommand{\C}{\mathcal C}
\newcommand{\IT}{\mathbb T}
\newcommand{\IR}{\mathbb R}

\newcommand{\Cyclic}{\mathsf{Cyclic}}

\newcommand{\Boolean}{\mathsf{Boolean}}
\newcommand{\Abelian}{\mathsf{Ab}}
\newcommand{\Ab}{\mathsf{Ab}}

\newcommand{\Group}{\mathsf{Group}}
\newcommand{\Dihedral}{\mathsf{Dihedral}}
\newcommand{\Spec}{\mathrm{Spec}}
\newcommand{\GR}{\mathrm{GR}}

\title{Difference bases in finite Abelian groups}
\author{Taras Banakh and Volodymyr Gavrylkiv}

\address[T.~Banakh]{Ivan Franko National University of Lviv (Ukraine),
and \newline Institute of Mathematics, Jan Kochanowski University in
Kielce (Poland)}
\email{t.o.banakh@gmail.com}
\address[V.~Gavrylkiv]{Vasyl Stefanyk Precarpathian National
University,
Ivano-Frankivsk, Ukraine} \email{vgavrylkiv@gmail.com}

\begin{document}

\begin{abstract} A subset $B$ of a group $G$ is called a {\em
difference basis} of $G$ if each element $g\in G$ can be written as the
difference $g=ab^{-1}$ of some elements $a,b\in B$. The smallest
cardinality $|B|$ of a difference basis $B\subset G$ is called the {\em
difference size} of $G$ and is denoted by $\Delta[G]$. The fraction $\eth[G]:=\frac{\Delta[G]}{\sqrt{|G|}}$ is called the {\em difference characteristic} of $G$.
Using properies of the Galois rings, we prove recursive upper bounds for the difference sizes and characteristics of finite Abelian groups. In particular, we prove that for a prime number $p\ge 11$, any finite Abelian $p$-group $G$ has difference characteristic $\eth[G]<\frac{\sqrt{p}-1}{\sqrt{p}-3}\cdot\sup_{k\in\IN}\eth[C_{p^k}]<\sqrt{2}\cdot\frac{\sqrt{p}-1}{\sqrt{p}-3}$. Also we calculate the difference sizes of all Abelian groups of  cardinality $<96$.
\end{abstract}
\maketitle

\section{Introduction}

A subset $B$ of a group $G$ is called a {\em difference basis} for a subset $A\subset G$ if each element $a\in A$ can be written as $a=xy^{-1}$ for some $x,y\in B$.
 The smallest cardinality of a difference basis for $A$ is called the {\em difference size} of $A$ and is denoted by $\Delta[A]$. For example, the set $\{0,1,4,6\}$ is a difference basis for the interval $A=[-6,6]\cap\IZ$ witnessing that $\Delta[A]\le 4$. The difference size is subadditive in the sense that $\Delta[A\cup B]<\Delta[A]+\Delta[B]$ for any non-empty finite subsets $A,B$ of a group $G$ (see Proposition~\ref{p:adic}).

The definition of a difference basis $B$ for a set $A$ in a group $G$ implies that $|A|\le |B|^2$ and hence $\Delta[A]\ge \sqrt{|A|}$. The fraction
$$\eth[A]:=\frac{\Delta[A]}{\sqrt{|A|}}\ge1$$is called the {\em difference characteristic} of $A$. The difference characteristic is submultiplicative in the sense that $\eth[G]\le \eth[H]\cdot\eth[G/H]$ for any normal subgroup $H$ of a finite group $G$, see \cite[1.1]{BGN}.

In this paper we are interested in evaluating the difference characteristics of
finite Abelian groups. In fact, this problem has been studied in the
literature. A fundamental result in this area is due to Kozma and Lev \cite{KL},
who proved (using the classification of finite simple groups) that each finite
group $G$ has difference characteristic $\eth[G]\le\frac{4}{\sqrt{3}}\approx
2.3094$. For finite cyclic groups $G$ the upper bound $\frac4{\sqrt{3}}$ can be
lowered to $\eth[G]\le \frac32$ (and even to $\eth[G]<\frac2{\sqrt{3}}$ if
$|G|\ge 2\cdot 10^{15}$), see \cite{BG}. In this paper we continue
investigations started in \cite{BG} and shall give some lower and upper bounds
for the difference characteristics of finite Abelian groups. In some cases (for
example, for Abelian $p$-groups with $p\ge11$) our upper bounds are better than
the general upper bound $\frac4{\sqrt{3}}$ of Kozma and Lev.

In particular, in Theorem~\ref{t:u-Abp} we prove that for any prime number $p\ge 11$, any finite Abelian $p$-group $G$ has difference characteristic $$\eth[G]\le \frac{\sqrt{p}-1}{\sqrt{p}-3}\cdot\sup_{k\in\IN}\eth[C_{p^k}]<
\frac{\sqrt{p}-1}{\sqrt{p}-3}\cdot\sqrt{2}.$$

These results are obtained exploiting a structure of a Galois ring on the groups $C_{p^k}^r$. Here $C_n=\{z\in\IC:z^n=1\}$ is the cyclic group of order $n$. The group $C_n$ is isomorphic to the additive group of the ring $\IZ_n=\IZ/n\IZ$. 

\section{Known results}

In this section we recall some known results on difference bases in
finite groups. The following fundamental fact was proved by Kozma and Lev
\cite{KL}.

\begin{theorem}[Kozma, Lev]\label{t:KL} Each finite group $G$ has
difference characteristic $\eth[G]\le\frac{4}{\sqrt{3}}$.
\end{theorem}
For a real number
 $x$ we put $$\lceil x\rceil=\min\{n\in\IZ:n\ge x\}\mbox{ and }\lfloor
x\rfloor=\max\{n\in\IZ:n\le x\}.$$

The following proposition is proved in \cite[1.1]{BGN}.

\begin{proposition}\label{p:BGN} Let $G$
be a finite group. Then
\begin{enumerate}
\item $ \frac{1+\sqrt{4|G|-3}}2\le \Delta[G]\le
\big\lceil\frac{|G|+1}2\big\rceil$,
\item $\Delta[G]\le \Delta[H]\cdot \Delta[G/H]$ and
$\eth[G]\le\eth[H]\cdot\eth[G/H]$ for any normal subgroup
$H\subset G$;
\item $\Delta[G]\le |H|+|G/H|-1$ for any subgroup $H\subset G$;
\end{enumerate}
\end{proposition}

Finite groups $G$ with $\Delta[G]=\big\lceil\frac{|G|+1}2\big\rceil$
were characterized in \cite{BGN} as follows.

\begin{theorem}[Banakh, Gavrylkiv, Nykyforchyn] \label{upbound}
For a finite group $G$
\begin{enumerate}
\item[(i)] $\Delta[G]=\big\lceil\frac{|G|+1}2\big\rceil>\frac{|G|}2$
if and only if $G$ is isomorphic to one of the groups:\newline
$C_1$, $C_2$, $C_3$, $C_4$, $C_2\times C_2$, $C_5$, $D_6$,
$(C_2)^3$;
\item[(ii)] $\Delta[G]=\frac{|G|}2$ if and only if $G$ is isomorphic
to one of the groups: $C_6$, $C_8$, $C_4\times C_2$, $D_8$,
$Q_8$.
\end{enumerate}
\end{theorem}

In this theorem by $D_{2n}$ we denote the dihedral group of cardinality
$2n$ and by $Q_8$ the 8-element group of quaternion units.
In \cite{BGN} the difference size $\Delta[G]$ was calculated for all groups $G$
of cardinality $|G|\le 13$.

{
\begin{table}[ht]
\caption{Difference sizes of groups of order $\le13$}\label{tab:BGN}
\begin{tabular}{|c|c|c|c|cc|cc|ccccc|}
\hline
$G$:& $C_2$& $C_3$ & $C_5$&$C_4$ &$C_2{\times}C_2$ &$C_6$& $D_6$ & $C_8$
&$C_2{\times}C_4$  & $D_8$ & $Q_8$& $(C_2)^3$\\
\hline
$\Delta[G]$&2&2&3&3&3&3&4&4&4&4&4&5\\
\hline
\hline
$G$:&$C_{7}$& $C_{11}$& $C_{13}$ &$C_9$&$C_3{\times}C_3$ &$C_{10}$&
$D_{10}$ & $C_{12}$ &$C_2{\times}C_6$ &$D_{12}$& $A_4$ & $C_3{\rtimes}
C_4$\\
\hline
$\Delta[G]$&3&4&4&4&4&4&4&4&5&5&5&5\\
\hline
\end{tabular}
\end{table}
}

An important role in evaluating the difference sizes of cyclic groups is due to
difference sizes of the order-intervals $[1,n]\cap\IZ$ in the additive group
$\IZ$ of integer numbers. For a natural number $n\in\IN$ by $\Delta[n]$ we
denote the difference size of the order interval $[1,n]\cap\IZ$ and by
$\eth[n]:=\frac{\Delta[n]}{\sqrt{n}}$ its difference characteristic. The
asymptotics of the sequence $(\eth[n])_{n=1}^\infty$ was studied by R\'edei and
R\'enyi \cite{RR}, Leech \cite{Leech} and Golay \cite{Golay} who eventually
proved that $$\sqrt{2+\tfrac4{3\pi}}<
\sqrt{2+\max_{0<\varphi<2\pi}\tfrac{2\sin(\varphi)}{\varphi+\pi}}\le
\lim_{n\to\infty}\eth[n]=\inf_{n\in\IN}\eth[n]\le
\eth[6166]=\frac{128}{\sqrt{6166}}<\eth[6]=\sqrt{\tfrac{8}3}.$$

In \cite{BG} the difference sizes of the order-intervals $[1,n]\cap\IZ$ were
applied to give upper bounds for the difference sizes of finite cyclic groups.

\begin{proposition}\label{p:c<n} For every $n\in\IN$ we get the upper bound $\Delta[C_n]\le\Delta\big[\lceil\frac{n-1}2\rceil\big]$, which implies that
$$\limsup_{n\to\infty}\eth[C_n]\le\frac1{\sqrt{2}}\inf_{n\in\IN}\eth[n]\le\frac{64}{\sqrt{3083}}<\frac{2}{\sqrt{3}}.$$
\end{proposition}

The following facts on the difference sizes of cyclic groups were proved in \cite{BG}.

\begin{theorem}[Banakh, Gavrylkiv]\label{t:cyclic} For any $n\in\IN$ the cyclic group $C_n$ has the difference characteristic:
\begin{enumerate}
\item $\eth[C_n]\le\eth[C_4]=\frac32$;
\item $\eth[C_n]\le\eth[C_2]=\eth[C_8]=\sqrt{2}$ if $n\ne 4$;
\item $\eth[C_n]\le\frac{12}{\sqrt{73}}<\sqrt{2}$ if $n\ge 9$;
\item $\eth[C_n]\le\frac{24}{\sqrt{293}}<\frac{12}{\sqrt{73}}$ if $n\ge 9$ and $n\ne 292$;
\item $\eth[C_n]<\frac2{\sqrt{3}}$ if $n\ge 2\cdot 10^{15}$.
\end{enumerate}
\end{theorem}

For some special numbers $n$ we have more precise upper bounds for $\Delta[C_n]$. We recall that a number $q$ is a {\em prime power} if $q=p^k$ for some prime number $p$ and some $k\in\IN$.

The following theorem was derived in \cite{BG} from the classical results of Singer \cite{Singer}, Bose, Chowla \cite{Bose}, \cite{Chowla} and Rusza \cite{Rusza}.

\begin{theorem} Let $p$ be a prime number and $q$ be a prime power.
Then
\begin{enumerate}
\item $\Delta[C_{q^2+q+1}]=q+1$;
\item $\Delta[C_{q^2-1}]\le q-1+\Delta[C_{q-1}]\le q-1+\frac{3}2\sqrt{q-1}$;
\item $\Delta[C_{p^2-p}]\le p-3+\Delta[C_{p}]+\Delta[C_{p-1}]\le
    p-3+\frac32(\sqrt{p}+\sqrt{p-1})$.
\end{enumerate}
\end{theorem}

The following table of difference sizes of cyclic groups $C_n$ for $\le 100$ is
taken from \cite{BG}.

\begin{table}[ht]
\caption{Difference sizes and characteristics of cyclic groups $C_n$ for
$n\le100$}\label{tab:cycl}
\begin{tabular}{|c|c|c||c|c|c||c|c|c||c|c|c|}
\hline
$n$   & \!$\Delta[C_n]$\! & $\eth[C_n]$&$n$   & \!$\Delta[C_n]$\! & $\eth[C_n]$&$n$   & \!$\Delta[C_n]$\! & $\eth[C_n]$&$n$   & \!$\Delta[C_n]$\! & $\eth[C_n]$\\
\hline
1 & 1 & 1 	&26 & 6 & 1.1766...\!\! & 51 & 8 & 1.1202...\!\!& 76 & 10 & 1.1470...\\
2 & 2 &1.4142...  &27 & 6 & 1.1547...\!\!      & 52 & 9 & 1.2480...\!\!& 77 & 10 & 1.1396...\!\!\\
3 & 2 &1.1547...  &28 & 6 & 1.1338...\!\! 	& 53 & 9 & 1.2362...\!\!& 78 & 10 & 1.1322...\!\!\\
4 & 3 &1.5  	  &29 & 7 & 1.2998...\!\! 	& 54 & 9 & 1.2247...\!\!& 79 & 10 & 1.1250...\!\!\\
5 & 3 &1.3416...  &30 & 7 & 1.2780...\!\! 	& 55 & 9 & 1.2135...\!\!& 80 & 11 & 1.2298...\!\!\\
6 & 3 & 1.2247... &31 & 6 & 1.0776...\!\! & 56 & 9 & 1.2026...\!\!& 81 & 11 & 1.2222...\!\!\\
7 & 3 & 1.1338... &32 & 7 & 1.2374...\!\! & 57 & 8 & 1.0596...\!\!& 82 & 11 & 1.2147...\!\!\\
8 & 4 & 1.4142... &33 & 7 & 1.2185...\!\! & 58 & 9 & 1.1817...\!\!& 83 & 11 & 1.2074...\!\!\\
9 & 4 & 1.3333... 	  &34 & 7 & 1.2004...\!\! & 59 & 9 & 1.1717...\!\!& 84 & 11 & 1.2001...\!\!\\
10 & 4 & 1.2649... &35 & 7 & 1.1832...\!\! & 60 & 9 & 1.1618...\!\!& 85 & 11 & 1.1931...\!\!\\
11 & 4 & 1.2060... &36 & 7 & 1.1666...\!\! 	& 61 & 9 & 1.1523...\!\!& 86 & 11 & 1.1861...\!\!\\
12 & 4 & 1.1547... &37 & 7 & 1.1507...\!\! 	& 62 & 9 & 1.1430...\!\!& 87 & 11 & 1.1793...\!\!\\
13 & 4 & 1.1094... &38 & 8 & 1.2977...\!\! 	& 63 & 9 & 1.1338...\!\!& 88 & 11 & 1.1726...\!\!\\
14 & 5 & 1.3363... &39 & 7 & 1.1208...\!\! 	& 64 & 9 & 1.125\!\!& 89 & 11 & 1.1659...\!\!\\
15 & 5 & 1.2909... &40 & 8 & 1.2649...\!\! 	& 65 & 9 & 1.1163...\!\!& 90 & 11 & 1.1595...\!\!\\
16 & 5 & 1.25     &41 & 8 & 1.2493...\!\! 	& 66 & 10 & 1.2309...\!\!& 91 & 10 & 1.0482...\!\!\\
17 & 5 & 1.2126... &42 & 8 & 1.2344...\!\! 	& 67 & 10 & 1.2216...\!\!& 92 & 11 & 1.1468...\!\!\\
18 & 5 & 1.1785... &43 & 8 & 1.2199...\!\! 	& 68 & 10 & 1.2126...\!\!& 93 & 12 & 1.2443...\!\!\\
19 & 5 & 1.1470... &44 & 8 & 1.2060...\!\! 	& 69 & 10 & 1.2038...\!\!& 94 & 12 & 1.2377...\!\!\\
20 & 6 & 1.3416... &45 & 8 & 1.1925...\!\! 	& 70 & 10 & 1.1952...\!\!& 95 & 12 & 1.2311...\!\!\\
21 & 5 & 1.0910... &46 & 8 & 1.1795...\!\! 	& 71 & 10 & 1.1867...\!\!& 96 & 12 & 1.2247...\!\!\\
22 & 6 & 1.2792... &47 & 8 & 1.1669...\!\! 	& 72 & 10 & 1.1785...\!\!& 97 & 12 & 1.2184...\!\!\\
23 & 6 & 1.2510... &48 & 8 & 1.1547...\!\! 	& 73 & 9 & 1.0533...\!\!& 98 & 12 & 1.2121...\!\!\\
24 & 6 & 1.2247... &49 & 8 & 1.1428...\!\! 	& 74 & 10 & 1.1624...\!\!& 99 & 12 & 1.2060...\!\!\\
25 & 6 & 1.2      &50 & 8 & 1.1313...\!\! 	& 75 & 10 & 1.1547...\!\!& 100 & 12 & 1.2\\
\hline
\end{tabular}
\end{table}

\section{A lower bound for the difference size}

In this section we prove a simple lower bound for the difference size of an arbitrary finite set in a group. This lower bound improves the lower bound given in Proposition~\ref{p:BGN}(1). For a group $G$ by $1_G$ we denote the unique idempotent of $G$.

\begin{theorem}\label{t:lower} Each finite subset $A$ of a group $G$ has difference size $$\Delta[A]\ge
\frac{1+\sqrt{4|A_{>2}|+8|A_2|+1}}2\ge\frac{1+\sqrt{4|A_{>1}|+1}
}2,$$ where $A_{>2}=\{a\in A:a^{-1}\ne a\}$, $A_2=\{a\in A:a^{-1}=a\ne 1_G\}$ and $A_{>1}=\{a\in A:a\ne 1_G\}$.
\end{theorem}

\begin{proof} Take a difference basis $B\subset G$ for the set $A$ of cardinality
$|B|=\Delta[A]$ and consider the map $\xi:B\times B\to G$,
$\xi:(x,y)\mapsto xy^{-1}$. Observe that for the unit $1_G$ of the group $G$ the preimage $\xi^{-1}(1_G)$ coincides with the diagonal $\{(x,y)\in B\times B:x=y\}$ of the square $B\times B$ and hence has cardinality
$|\xi^{-1}(e)|=|B|$. Observe also that for any element $g\in A_2=\{a\in A:a^{-1}=a\ne 1_G\}$ and
any $(x,y)\in \xi^{-1}(g)$, we get $yx^{-1}=(xy^{-1})^{-1}=g^{-1}=g$,
which implies that $|\xi^{-1}(g)|\ge 2$.
Then
$$|B|^2=|B\times B|\ge|\xi^{-1}(1_G)|+\sum_{a\in
A_2}|\xi^{-1}(g)|+\sum_{a\in A_{>2}}|\xi^{-1}(g)|\ge |B|+2|A_2|+|A_{>2}|$$and
hence
$$\Delta[G]=|B|\ge
\frac{1+\sqrt{1+4|A_{>2}|+8|A_2|}}2\ge\frac{1+\sqrt{1+4|A_{>1}|}}2$$
as $A_{>2}\cup A_2=A_{>1}$.
\end{proof}

\begin{corollary}\label{c:lower} Each finite group $G$ has difference size $\Delta[G]\ge
\frac{1+\sqrt{4|G|+4|G_2|-3}}2$, where $G_2=\{g\in G:g^{-1}=g\ne 1_G\}$ is the set of elements of order 2 in $G$.
\end{corollary}

\section{The subadditivity and submultiplicativity of the difference size}

In this section we prove two properties of the difference size called
the subadditivity and the submultiplicativity.

\begin{proposition}\label{p:adic} For any non-empty finite subsets $A,B$ of a group $G$ we get $\Delta[A\cup B]\le\Delta[A]+\Delta[B]-1$.
\end{proposition}

\begin{proof}

Given non-empty sets $A,B\subset G$, find difference bases $D_A$ and $D_B$ for the sets $A,B$ of
cardinality $|D_A|=\Delta[A]$ and $|D_B|=\Delta[B]$. Taking any point $d\in D_A$ and replacing $D_A$
by its shift $D_Ad^{-1}$, we can assume that the unit $1_G$ of the group $G$ belongs to $D_A$.
By the same reason, we can assume that $1_G\in D_B$. The union $D=D_A\cup D_B$ is a difference
basis for $A\cup B$, witnessing that $$\Delta[A\cup B]\le |D|\le |D_A|+|D_B|-1=\Delta[A]+\Delta[B]-1.$$

\end{proof}

\begin{proposition}\label{p:multic} Let $h:G\to H$ be a surjective homomorphism of groups with finite kernel $K$.
For any non-empty finite subset $A\subset H$ we get $\Delta[h^{-1}(A)]\le \Delta[A]\cdot\Delta[K]$.
\end{proposition}

\begin{proof}
Given a non-empty finite subset $A\subset H$, find a difference basis $D_A$  for the set $A$ of
cardinality $|D_A|=\Delta[A]$. Also fix a difference basis $D_K$  for the kernel $K\subset G$ of
cardinality $|D_K|=\Delta[K]$.

Fix any subset $B\subset G$ such that $|B|=|D_A|$ and $|h^{-1}(x)\cap B|=1$ for any $x\in D_A$.
We claim that the set $C=BD_K$ is a difference basis for $h^{-1}(A)$.

Since $D_A$ is a difference basis   for the set $A$, for  any $a\in h^{-1}(A)$ there are elements $a_1,a_2\in D_A$ such that $h(a)=a_1a_2^{-1}$. Then $a=b_1b_2^{-1}k$ for some  $b_1,b_2\in B$, $k\in K$. The normality of the subgroup $K$ in $G$ implies that
$a=b_1b_2^{-1}k=b_1k'b_2^{-1}$ for some $k'\in K$. Taking into account that $D_K$ is a difference basis   for the
kernel $K$, find elements $k_1,k_2\in D_K$ such that $k'=k_1k_2^{-1}$.
Then $$a=b_1b_2^{-1}k=b_1k'b_2^{-1}=b_1k_1k_2^{-1}b_2^{-1}=(b_1k_1)(b_2k_2)^{-1}\in CC^{-1}$$ and
$$\Delta[h^{-1}(A)]\le |C|\le |D_A|\cdot|D_K|=\Delta[A]\cdot\Delta[K].$$
\end{proof}

\section{Difference bases in rings}

In this section we construct difference bases for subsets of finite rings. All
rings considered in this section have the unit. For a ring $R$ by $U(R)$ we
denote the multiplicative group of invertible elements in $R$. An element $x$
of a ring $R$ is called {\em invertible} if there exists an element $x^{-1}\in
R$ such that $xx^{-1}=x^{-1}x=1$. The group $U(R)$ is called {\em the group of units}
of the ring $R$.

The {\em characteristic} of a finite ring $R$ is the smallest natural number
$n$ such that $nx=0$ for every $x\in R$. A non-empty subset $I$ of a ring $R$
is called an {\em ideal} in $R$ if $I\ne R$, $I-I\subset I$ and $IR\cup
RI\subset I$. The spectrum $\Spec(R)$ of a ring is the set of all maximal
ideals of $R$. For any maximal ideal $I$ of a commutative ring $R$ the quotient
ring $R/I$ is a field.

A ring $R$ is called {\em local} if it contains a unique maximal ideal, which is denoted by $I_{\mathfrak m}$. The quotient ring $R/I_{\mathfrak m}$ is a field called the {\em residue field} of the local ring $R$.

By \cite[Ch.6]{BF}, for every prime number $p$ and natural numbers $k,r$ there exists a unique local ring $\GR(p^k,r)$ called the {\em Galois ring} of characteristic $p^k$ whose additive group is isomorphic to the group $(C_{p^k})^r$, the maximal ideal $I_{\mathfrak m}$ coincides with the principal ideal $pR$ generated by $p\cdot 1$ and whose residue field $\GR(p^k,r)$ contains $p^r$ elements.  The Galois ring $\GR(p^k,r)$ can be constructed as the quotient ring $\IZ[x]/(p^k,f(x))$ of the ring $\IZ[x]$ of polynomials with integer coefficients by the ideal $(p^k,f(x))$ generated by the constant $p^k$ and a carefully selected monic polynomial $f\in\IZ[x]$ of degree $r$, which is irreducible over the field $\IZ/p\IZ$, see \cite[6.1]{BF}. For $k=1$ the Galois ring $\GR(p^k,r)$ is a field, and for $r=1$ the Galois ring $\GR(p^k,r)$ is isomorphic to the ring $\IZ/p^k\IZ$.

The following description of the multiplicative group of a Galois ring is taken from Theorem 6.1.7 of the book \cite{BF}.

\begin{theorem}\label{t:R*} Let $p$ be a prime number and $k,r$ be natural numbers. The multiplicative group $U(R)$ of a Galois ring $R:=\GR(p^k,r)$ is isomorphic to:
$$\begin{cases}
C_{p^r-1}\times C_{p^{k-1}}^r&\mbox{if either $p$ is odd or $p=2$ and $k\le 2$},\\
C_{2^r-1}\times C_2\times C_{2^{k-2}}\times C_{2^{k-1}}^{r-1}&\mbox{if $p=2$ and $k\ge 3$}.
\end{cases}
$$
\end{theorem}


The following theorem is our principal tool for evaluating the difference sizes of finite Abelian groups of odd order.

\begin{theorem}\label{t:xx} Let $R$ be a finite commutative ring $R$ with unit and $(1+1)\in U(R)$ and let $h:G\to R\times R$ be a surjective homomorphism from a group $G$ onto the Abelian group of the ring $R\times R$. Let $K=h^{-1}(0,0)$ be the kernel of the homomorphism $h$.
Then $$\Delta[G]\le \Delta[K]\cdot |R|-|\Spec(R)|+\sum_{I\in\Spec(R)}\Delta[h^{-1}(I\times R)].$$If the ring $R$ is local, then
$\Delta[G]\le \Delta[K]\cdot |R|-1+\Delta[h^{-1}(I_{\mathfrak m}\times R)].$
\end{theorem}

\begin{proof} First we observe that $R=U(R)\cup\bigcup_{I\in\Spec(R)}I$. Indeed, if an element $x\in R$ is not invertible, then the set $xR=\{xy:y\in R\}$ is an ideal in $R$, contained in some maximal ideal $I\in\Spec(R)$. This implies that $R=U(R)\cup\bigcup_{I\in\Spec(R)}I$ and hence $G=h^{-1}(U(R)\times R)\cup\bigcup_{I\in\Spec(R)}h^{-1}(I\times R)$.

\begin{lemma}\label{l:xx} The set $B=\{(x,x^2):x\in R\}$ is a difference basis for the set $U(R)\times R$ in the additive group $R\times R$.
\end{lemma}

\begin{proof} Given any pair $(a,b)\in U(R)\times R$, we should find two elements $x,y\in R$ such that $(a,b)=(x-y,x^2-y^2)$. Solving this system of equations in the ring $R$, we get the solution $$
\begin{cases}
x=2^{-1}(ba^{-1}+a)\\
y=2^{-1}(ba^{-1}-a).
\end{cases}
$$
\end{proof}

By Lemma~\ref{l:xx}, the set $B=\{(x,x^2):x\in R\}$ is a difference basis for the set $U(R)\times R$ in $R\times R$. So, $\Delta[U(R)\times R]\le|B|=|R|$. By Proposition~\ref{p:multic}, $$h^{-1}(U(R)\times R)\le\Delta[K]\cdot\Delta[U(R)\times R]\le\Delta[K]\cdot |R|$$and by Proposition~\ref{p:adic},
$$\Delta[G]\le \Delta[h^{-1}(U(R)\times R)]+\sum_{I\in\Spec(R)}(\Delta[h^{-1}(I\times R)]-1)\le
\Delta[K]\cdot |R|-|\Spec(R)|+\sum_{I\in\Spec(R)}\Delta[h^{-1}(I\times R)].
$$
\end{proof}

Our next theorem will be applied for evaluating the difference characteristics of Abelian 2-groups. This theorem exploits the structure of a (non)associative ring. By a {\em (non)associative ring} we understand an Abelian group $R$ endowed with a binary operation $\circ:R\times R\to R$ which is distributive in the sense that $x\circ(y+z)=x\circ y+x\circ z$ and $(x+y)\circ z=x\circ z+y\circ z$ for all $x,y,z\in R$. A (non)associative ring $R$ is called a {\em ring} if its binary operation $\circ$ is associative. In the opposite case $R$ is called a {\em non-associative ring}.

For any (non)associative ring $R$ the product $R\times R$, endowed with the binary operation $$(x,y)\star(x',y')=(x+x',y+y'+x\circ x'),$$is a group. The inverse element to $(x,y)$ in this group is $(-x,-y+x\circ x)$. The product $R\times R$ endowed with this group operation will be denoted by $R\star R$. The group $R\star R$ is commutative if and only if the binary operation $\circ$ on $R$ is commutative. For a (non)associative ring $R$ let $U(R)$ be the set of all elements $a\in R$ such that the map $R\to R$, $x\mapsto x\circ a$, is bijective. The following theorem was known for semifields, see \cite[4.1]{PSZ}.

\begin{theorem}\label{t:RR} For any finite (non)associative ring $R$ the set $B=\{(x,x\circ x):x\in R\}$ is a difference basis for the set $U(R)\times R$ in the group $R\star R$.
\end{theorem}

\begin{proof} Given any pair $(a,b)\in U(R)\times R$, we need to find elements $x,y\in R$ such that $(a,b)=(x,x\circ x)\star(y,y\circ y)^{-1}$. The definition of the group operation $\star$ implies that $(y,y\circ y)^{-1}=(-y,0)$. Then the equality $(a,b)=(x,x\circ x)\star(y,y\circ y)^{-1}$ turns into $(a,b)=(x,x\circ x)\star(-y,0)=(x-y,x\circ x-x\circ y)=(x-y,x\circ(x-y))$. Since $a\in U(R)$, there exists an element $x\in R$ such that $x\circ a=b$. Let $y=x-a$ and observe that the pair $(x,y)$ has the required property:
$$(x,x\circ x)\star(y,y\circ y)^{-1}=(x-y,x\circ (x-y))=(a,x\circ a)=(a,b).$$
\end{proof}

Theorem~\ref{t:RR} suggests the problem of detecting the structure of the group $R\star R$ for various rings $R$. For Galois rings $GR(p^k,r)$ this problem is answered in the following two theorems.

\begin{theorem}\label{t:RstarR} Let $p$ be a prime number, and $k,r$ be natural numbers. For the Galois ring $R:=GR(p^k,r)$ the group $R\star R$ is isomorphic to
$$
\begin{cases}
C_{p^k}^{r}\times C_{p^k}^r&\mbox{ if \ $p\ge 3$},\\
C_{2^{k+1}}^r\times C_{2^{k-1}}^r&\mbox{ if \ $p=2$}.
\end{cases}
$$
\end{theorem}

\begin{proof} First observe that the commutativity of the Galois ring $R$ implies the commutativity of the group $R\star R$. To determine the structure of the group $R\star R$ we shall calculate the orders of its  elements. Let us recall that the {\em order} of an element $x$ in an Abelian group $G$ is the smallest number $n\in\IN$ such that $nx=0$.

Let us fix an element $(x,y)\in R\times R$ and evaluate its order in the group $R\star R$. By induction it can be shown that for every $s\in\IN$ we get $(x,y)^s=(sx,sy+\frac{s(s-1)}2x^2)$.

If $p\ge 3$, then $(x,y)^{p^k}=\big(p^kx,p^ky+p^k\frac{p^k-1}2x^2\big)=(0,0)$ as $p^k\cdot z=0$ for each element $z\in R$. On the other hand, $(x,y)^{p^{k-1}}=(0,0)$ if and only if $p^{k-1}x=p^{k-1}y=0$ if and only if $(x,y)\in pR\times pR$, which implies that the set of elements of order $p^{k-1}$ has cardinality $p^{2(k-1)r}$. It remains to observe that up to an isomorphism, $C_{p^k}^{2r}$ is the unique Abelian $p$-group of cardinality $p^{2kr}$ that contains $p^{2kr}$ elements of order
$\le p^{k}$ and $p^{2kr}-p^{2(k-1)r}$ elements of order $p^k$.
\smallskip

Next, assume that $p=2$. In this case $(x,y)^{2^{k+1}}=(2^{k+1}x,2^{k+1}y+2^k(2^{k+1}-1)x^2)=(0,0)$, which means that each element of the group $R\star R$ has order $\le 2^{k+1}$. Observe that $(x,y)^{2^k}=(2^kx,2^ky+2^{k-1}(2^k-1)x^2)=(0,2^{k-1}(2^k-1)x^2)$, which implies that $(x,y)^{2^k}\ne (0,0)$ if and only if $x^2\notin I_{\mathfrak m}=2R$ if and only if $x\in U(R)$. This means that the 2-group $R\star R$ has exactly $|U(R)\times R|=(2^{kr}-2^{(k-1)r})2^{kr}=2^{(2k-1)r}(2^r-1)$ elements of order $2^{k+1}$.

Next, for any $i\in \{0,\dots,k\}$, we calculate the number of elements of order $>2^{k-i}$ in $R\star R$. Observe that an element $(x,y)\in R\star R$ has order $>2^{k-i}$ if and only if $(2^{k-i}x,2^{k-i}y+2^{k-i-1}(2^{k-i}-1)x^2)\ne (0,0)$ if and only if either $x\notin 2^{i}R$ or $x\in 2^{i}R$ and $y\notin 2^{i}R$. So, the set of elements of order $>2^{k-i}$ coincides with $\big((R\setminus 2^{i}R)\times R\big)\cup \big(2^{i}R\times (R\setminus 2^{i}R)\big)$ and hence has cardinality $$
|R\setminus 2^{i}R|\cdot |R|+|2^{i}R|\cdot|R\setminus 2^{i}R|=(|R|-|2^{i}R|)\cdot(|R|+|2^{i}R|)=|R|^2-|2^{i}R|^2=
2^{2kr}-2^{2(k-i)r}=
2^{2(k-i)r}(2^{2ir}-1).
$$

This information is sufficient to detect the isomorphic type of the group
$R\star R$. By \cite[4.2.6]{Rob}, the Abelian 2-group $R\star R$ is isomorphic
to the product $H=\prod_{i=1}^{k+1}C_{2^{i}}^{m_i}$ for some numbers
$m_1,\dots, m_{k+1}\in\{0\}\cup\IN$. Observe that the group $H$ contains
$(2^{(k+1)m_{k+1}}-2^{km_{k+1}})\cdot
\prod_{i=1}^k2^{im_i}=2^{km_{k+1}}(2^{m_{k+1}}-1)\cdot\prod_{i=1}^k2^{im_i}$
elements of order $2^{k+1}$. Taking into account that the group $R\star R$
contains $2^{(2k-1)r}(2^{r}-1)$ elements of order $2^{k+1}$, we conclude that
$m_{k+1}=r$.

Next, observe that the group $H$ contains exactly
\begin{multline*}
|C_{2^{k+1}}^{m_{k+1}}\times C_{2^k}^{m_k}-C_{2^{k-1}}^{m_{k+1}}\times C_{2^{k-1}}^{m_k}|\cdot\prod_{i=1}^{k-1}|C_{2^i}^{m_i}|=\\
=
(2^{(k+1)r+km_k}-2^{(k-1)(r+m_k)})\cdot\prod_{i=1}^{k-1}2^{im_i}=
2^{(k-1)(r+m_k)}(2^{2r+m_k}-1)\cdot\prod_{i=1}^{k-1}2^{im_i}
\end{multline*}
elements of order $\ge 2^k$.
Taking into account that the group $R\star R$ contains exactly
$2^{2(k-1)n}(2^{2r}-1)$ elements of order $\ge 2^k$, we conclude that $m_k=0$.

The group $H$ contains exactly
\begin{multline*}
|C_{2^{k+1}}^{m_{k+1}}\times C_{2^k}^{m_k}\times C_{2^{k-1}}^{m_{k-1}}-C_{2^{k-2}}^{m_{k+1}}\times C_{2^{k-2}}^{m_k}\times C_{2^{k-2}}^{m_{k-1}}|\cdot\prod_{i=1}^{k-2}|C_{2^i}^{m_i}|=\\
=
(2^{(k+1)r+(k-1)m_{k-1}}-2^{(k-2)(r+m_{k-1})})
\cdot\prod_{i=1}^{k-2}2^{im_i}=
2^{(k-2)(r+m_{k-1})}(2^{3r+m_{k-1}}-1)\cdot\prod_{i=1}^{k-2}2^{im_i}
\end{multline*}
elements of order $\ge 2^{k-1}$.
Taking into account that the group $R\star R$ contains exactly
$2^{(k-2)r}(2^{4r}-1)$ elements of order $\ge 2^{k-1}$, we conclude that $m_{k-1}=r$.

Taking into account that $|C_{2^{k+1}}^{m_{k+1}}\times C_{2^{k-1}}^{m_{k-1}}|=|C_{2^{k+1}}^r\times C_{2^{k-1}}^r|=2^{2kr}=|R\star R|$, we conclude that $m_i=0$ for $i<k-1$ and hence the group $R\star R$ is isomorphic to $C_{2^{k+1}}^{r}\times C_{2^{k-1}}^{r}$.
\end{proof}

\section{Evaluating the difference characteristics of Abelian $p$-groups}

In this section we shall evaluate the difference characteristics of  finite Abelian $p$-groups for an odd prime number $p$. We recall that a group $G$ is called a {\em $p$-group} if each element $x\in G$ generates a finite cyclic group of order $p^k$ for some $k\in\IN$. A finite group $G$ is a $p$-group if and only if $|G|=p^k$ for some $k\in\IN$.

It is well-known that each Abelian $p$-group $G$ is isomorphic to the product $\prod_{i=1}^rC_{p^{k_i}}$ of cyclic $p$-groups. The number $r$ of cyclic groups in this decomposition is denoted by $r(G)$ and called the {\em rank} of $G$.

Applying Theorem~\ref{t:xx} to the Galois ring $R:=\GR(p^k,r)$ and taking into account that its additive group is isomorphic to $C_{p^k}^r$ and $pR$ coincides with the maximal ideal of $R$, we get the following corollary.

\begin{corollary}\label{c:hxx} Let $p$ be an odd prime number, $k,r$ be natural numbers. Let $h:G\to C_{p^k}^{2r}$ be a surjective homomorphism and $K$ be its kernel. Then
$$\Delta[G]\le\Delta[K]\cdot p^{kr}+\Delta[h^{-1}(C_{p^{k-1}}^r\times C_{p^k}^r)]-1$$and
$$\eth[G]\le\eth[K]+\frac1{\sqrt{p^r}}\cdot\eth[h^{-1}(C_{p^{k-1}}^r\times C_{p^k}^r)]-\frac1{\sqrt{|G|}}\,.$$
\end{corollary}

This corollary implies the following recursive upper bound for difference characteristics of finite Abelian $p$-groups.

\begin{theorem}\label{t:Abp} Let $p$ be an odd prime number, $k_1,\dots,k_m$ be natural numbers, and $k,r$ be natural numbers such that $2r\le m$ and $k\le\min\limits_{1\le i\le 2r}k_i$. Then
$$\Delta\Big[\prod_{i=1}^mC_{p^{k_i}}\Big]\le \Delta\Big[\prod_{i=1}^{2r}C_{p^{k_i-k}}\times
\prod_{i=2r+1}^mC_{p^{k_i}}\Big]\cdot p^{kr}+\Delta\Big[\prod_{i=1}^r
C_{p^{k_i-1}}\times
\prod_{i=r+1}^mC_{p^{k_i}}\Big]-1$$and
$$\eth\big[\prod_{i=1}^mC_{p^{k_i}}\big]\le \eth\big[\prod_{i=1}^{2r}C_{p^{k_i-k}}\times
\prod_{i=2r+1}^mC_{p^{k_i}}\big]+\frac1{\sqrt{p^r}}\cdot\eth\big[\prod_{i=1}^r
C_{p^{k_i-1}}\times
\prod_{i=r+1}^mC_{p^{k_i}}\big]-\prod_{i=1}^m\frac1{\sqrt{p^{k_i}}}.$$
\end{theorem}

The recursive formulas from the preceding theorem will be used in the following upper bound for the difference characteristic of finite Abelian $p$-group.

\begin{theorem}\label{t:u-Abp} For any prime number $p\ge 11$, any finite Abelian $p$-group $G$ has difference
characteristic
$$\eth[G]\le \frac{\sqrt{p}-1}{\sqrt{p}-3}\cdot\sup_{k\in\IN}\eth[C_{p^k}]\le
\frac{\sqrt{p}-1}{\sqrt{p}-3}\cdot\frac{24}{\sqrt{293}}.
$$
\end{theorem}

\begin{proof} For a prime number $p$ and a natural number $r$ let $\Ab_p^r$ be the class of Abelian $p$-groups of rank $r$. Let also $\Ab_{p}^{<r}=\bigcup_{n<r}\Ab_{p}^n$ and
$\Ab_{p}^{\le r}=\bigcup_{n\le r}\Ab_{p}^n$.
Let $\Ab_{p}:=\bigcup_{r\in\IN}\Ab_{p}^r$ be the family of finite Abelian $p$-groups.
For a class $\C$ of finite groups we put $\eth[\C]:=\sup_{G\in\C}\eth[G]$. By Theorem~\ref{t:KL}, $\eth[\C]\le\frac4{\sqrt{3}}$.

\begin{lemma}\label{l:Ab} For any odd prime number $p$ and any natural number $r$ we get the upper bound
$$\eth[\Ab_p^{r}]\le \eth[\Ab_{p}^{<r}]+\frac1{\sqrt{p^{\lfloor r/2\rfloor}}}\cdot
\eth[\Ab_p^{\le r}]\le \eth[\Ab_{p}^{<r}]+\frac1{\sqrt{p^{\lfloor r/2\rfloor}}}\cdot
\eth[\Ab_p].$$
Consequently, $$\eth[\Ab_p^r]\le\eth[\Ab_p^1]+\eth[\Ab_p]\cdot\sum_{i=2}^r\frac1{\sqrt{p^{\lfloor i/2\rfloor}}}\le\eth[\Ab_p^1]+\eth[\Ab_p]\cdot\sum_{i=1}^{\lfloor r/2\rfloor}\frac2{\sqrt{p^i}}$$
\end{lemma}

\begin{proof} Any group $G\in\Ab^{r}$ is isomorphic to the product $\prod_{i=1}^{r}C_{p^{k_i}}$ for a unique non-decreasing sequence $(k_i)_{i=1}^{r}$ of natural numbers. Let $k=k_1$ and $m=\lfloor\frac{r}2\rfloor$. Since $k_1-k=0$, the group $\prod_{i=1}^{2m}C_{p^{k_i-k}}\times\prod_{i=2m+1}^rC_{p^{k_i}}$ has rank $<r$. Applying Theorem~\ref{t:Abp}, we conclude that
$$
\begin{aligned}
\eth[G]&=\eth\big[\prod_{i=1}^rC_{p^{k_i}}\big]< \eth\big[\prod_{i=1}^{2m}C_{p^{k_i-k}}\times
\prod_{i=2m+1}^rC_{p^{k_i}}\big]+\frac1{\sqrt{p^m}}\cdot
\eth\big[\prod_{i=1}^m
C_{p^{k_i-1}}\times
\prod_{i=m+1}^rC_{p^{k_i}}\big]\le\\
&\le \eth[\Ab_p^{<r}]+\frac1{\sqrt{p^m}}
\cdot\eth[\Ab_{p}^{\le r}]\le \eth[\Ab_p^{<r}]+\frac1{\sqrt{p^m}}
\cdot\eth[\Ab_{p}].
\end{aligned}
$$
\end{proof}

Lemma~\ref{l:Ab} implies that
$$\eth[\Ab_p]\le \eth[\Ab_p^1]+\eth[\Ab_p]\cdot\sum_{i=1}^\infty\frac2{\sqrt{p^i}}=\eth[\Ab_p^1]
+\eth[\Ab_p]\cdot \frac{2}{\sqrt{p}-1}$$and after transformations
$$\eth[\Ab_p]\le\eth[\Ab_p^1]\cdot\Big(1-\frac2{\sqrt{p}-1}\Big)^{-1}=
\eth[\Ab_p^1]\cdot\frac{\sqrt{p}-1}{\sqrt{p}-3}=
\frac{\sqrt{p}-1}{\sqrt{p}-3}\cdot\sup_{k\in\IN}\eth[C_{p^k}]\le \frac{\sqrt{p}-1}{\sqrt{p}-3}\cdot\frac{24}{\sqrt{293}}\;.$$
In the last inequality we use the upper bound $\sup_{k\in\IN}\eth[C_{p^k}]\le\frac{24}{\sqrt{293}}$ from Theorem~\ref{t:cyclic}.
\end{proof}

Theorem~\ref{t:Abp} implies:

\begin{corollary} For any odd prime number $p$ and natural numbers $k,n$ the groups $C_{p^k}^{2n}$ and $C_{p^k}^{2n+1}$ have difference characteristics
$$\eth[C_{p^k}^{2n}]\le 1-\frac1{p^{kr}}+\frac1{\sqrt{p^r}}\cdot\eth[C_{p^{k-1}}^n\times C_{p^k}^n]<1+\frac1{\sqrt{p^r}}\cdot\eth[\Ab_p]$$and
$$\eth[C_{p^k}^{2n+1}]\le \eth[C_{p^k}]-\frac1{p^{kr}}+\frac1{\sqrt{p^r}}\cdot\eth[C_{p^{k-1}}^n\times C_{p^k}^{n+1}]<\eth[C_{p^k}]+\frac1{\sqrt{p^r}}\cdot\eth[\Ab_p].$$
\end{corollary}

\section{Evaluating the difference characteristics of $2$-groups}

In this section we elaborate tools for evaluating the difference characteristics of 2-groups. The following corollary is a counterpart of Corollary~\ref{c:hxx}.

\begin{corollary} Let $k,r$ be natural numbers. Let $h:G\to C_{2^{k+1}}^{r}\times C_{2^{k-1}}^r$ be a surjective homomorphism and $K$ be its kernel. Then
$$\Delta[G]\le\Delta[K]\cdot 2^{kr}+\Delta[h^{-1}(C_{2^{k}}^r\times C_{2^{k-1}}^r)]-1$$and
$$\eth[G]\le\eth[K]+\frac1{\sqrt{2^r}}\cdot
\eth[h^{-1}(C_{2^{k}}^r\times C_{2^{k-1}}^r)]-\frac1{\sqrt{|G|}}\,.$$
\end{corollary}

\begin{proof} Consider the Galois ring $R:=\GR(2^k,r)$, whose additive group is isomorphic to $C_{2^r}^k$. Its maximal ideal $I_{\mathfrak m}$ coincides with the subgroup $2R$ of $R$, which consists of elements of order $\le 2^{k-1}$ in $R$. The subset $pR\times R$ of the group $R\star R$ has cardinality $2^{(2k-1)r}$ and consists of elements of order $\le 2^{k}$ in $R\star R$. By Theorem~\ref{t:RstarR}, the group $R\star R$ is isomorphic to $C_{2^{k+1}}^r\times C_{2^{k-1}}^r$. It is easy to see that $C_{2^k}^r\times C_{2^{k-1}}^r$ is the unique subgroup of cardinality $2^{(2k-1)r}$ consisting of elements of order $\le 2^{k}$.
Therefore, the group $C_{2^{k+1}}^r\times C_{2^{k-1}}^r$ can be identified with the group $R\star R$ and its subgroup $C_{2^k}^r\times C_{2^{k-1}}^r$ with the subgroup $2R\times R$ of the group $R\times R$. By Theorem~\ref{t:RR}, the set $B=\{(x,x^2):x\in R\}$ is a difference base for the set $U(R)\times R$ in the group $R\times R$. So, $\Delta[U(R)\times R]\le|R|=2^k$. By Propositions~\ref{p:adic} and \ref{p:multic},
$$\Delta[G]\le\Delta[K]\cdot\Delta[U(R)\times R]-1+\Delta[h^{-1}(I_{\mathfrak m}\times R)]\le\Delta[K]\cdot |R|-1+\Delta[h^{-1}(2R\times R)]$$ and hence
$$
\eth[G]=\frac{\Delta[K]\cdot |R|}{\sqrt{|K|\cdot|R|^2}}-\frac1{\sqrt{|G|}}+\frac{\Delta[h^{-1}(2R\times R)]}{\sqrt{|K|\cdot|2R|\cdot|R|\cdot|R/2R|}}=\eth[K]-\frac1{\sqrt{|G|}}+\frac1{\sqrt{2^r}}\cdot\eth[h^{-1}(2R\times R)].
$$
\end{proof}

This corollary implies the following recursive upper bound for difference characteristics of finite Abelian $2$-groups.

\begin{theorem}\label{t:Ab2} Let $k_1,\dots,k_m$ be natural numbers, and $k,r$ be natural numbers such that $2r\le m$, $k+1\le\min\limits_{1\le i\le r}k_i$ and $k-1\le\min\limits_{r<i\le 2r}k_i$. Then
$$\Delta\Big[\prod_{i=1}^mC_{2^{k_i}}\Big]\le \Delta\Big[\prod_{i=1}^{r}C_{2^{k_i-k-1}}\times
\prod_{i=r+1}^{2r}C_{2^{k_i-k+1}}\times
\prod_{i=2r+1}^mC_{2^{k_i}}\Big]\cdot 2^{kr}+\Delta\Big[\prod_{i=1}^r
C_{2^{k_i-1}}\times
\prod_{i=r+1}^mC_{2^{k_i}}\Big]-1$$and
$$\eth\Big[\prod_{i=1}^mC_{2^{k_i}}\Big]\le \eth\Big[\prod_{i=1}^{r}C_{2^{k_i-k-1}}\times
\prod_{i=r+1}^{2r}C_{2^{k_i-k+1}}\times
\prod_{i=2r+1}^mC_{2^{k_i}}\Big]+\frac1{\sqrt{2^r}}\cdot\eth
\Big[\prod_{i=1}^r
C_{2^{k_i-1}}\times
\prod_{i=r+1}^mC_{2^{k_i}}\Big]-\prod_{i=1}^m\frac1{2^{k_i}}.$$
\end{theorem}

Now we shall evaluate the difference characteristics of the 2-groups $C_{2^n}^r$.

\begin{proposition}\label{p:Boolean} For any $n\in\IN$ the groups $C_2^{2n}$ and $C_2^{2n+1}$ have difference sizes
$$\frac{1+\sqrt{2^{2n+3}-7}}2\le \Delta[C_2^{2n}]<2^{n+1}\mbox{ \ and \ }\frac{1+\sqrt{2^{2n+4}-7}}2\le \Delta[C_2^{2n+1}]<3\cdot 2^{n}$$and difference characteristics
$$\sqrt{2}<\frac{1+\sqrt{2^{2n+3}-7}}{2^{n+1}}\le \eth[C_2^{2n}]<2\mbox{ \ and \ }\sqrt{2}<\frac{1+\sqrt{2^{2n+4}-7}}{\sqrt{2}\cdot 2^{n+1}}\le \eth[C_2^{2n+1}]<\frac{3}{\sqrt{2}}.$$
\end{proposition}

\begin{proof} The lower bound $\frac{1+\sqrt{8|G|-7}}2\le \Delta[G]$ follows from Theorem~\ref{t:lower}.
\smallskip

The upper bound will be derived from Proposition~\ref{p:BGN}(3) which implies that
$$\Delta[C_2^{2n}]<|C_2^n|+|C_2^n|=2\cdot 2^n$$
and
$$\Delta[C_2^{2n+1}]<|C_2^n|+|C_2^{n+1}|=3\cdot 2^n.$$
These upper bounds imply
$$\eth[C_2^{2n}]<2\mbox{ \ and \ }\eth[C_2^{2n+1}]<\frac{3}{\sqrt{2}}.$$
\end{proof}

\begin{proposition}\label{p:C4n} For any $n\in\IN$ the group $C_4^n$ has the difference characteristic $$\eth[C_4^n]\le 1+\frac1{\sqrt{2^n}}\cdot\eth[C_2^n]-\frac1{2^n}< 1-\frac1{2^n}+\frac3{\sqrt{2^{n+1}}}.$$
\end{proposition}

\begin{proof} Consider the numbers $k_1=\dots=k_n=2$ and $k_{n+1}=\dots=k_{2n}=1$. By Theorem~\ref{t:Ab2},
$$\eth[C_4^n]=\eth\Big[\prod_{i=1}^{2n}C_{2^{k_i}}\big]\le\eth[C_1]-\frac1{2^n}+
\frac1{\sqrt{2^n}}\eth\Big[\prod_{i=1}^nC_{2^{k_i-1}}\times\prod_{i=n+1}^{2n}C_{2^{k_i}}\Big]= 1-\frac1{2^n}+\frac1{\sqrt{2^n}}\eth[C_{2}^n]<1-\frac1{2^n}+\frac{3}{\sqrt{2^{n+1}}}.
$$In the last inequality we used the upper bound $\eth[C_2^n]<\frac3{\sqrt{2}}$, proved in Proposition~\ref{p:Boolean}.
\end{proof}

Theorem~\ref{t:Ab2}, Proposition~\ref{p:C4n} and Theorem~\ref{t:KL} imply:

\begin{corollary} For any $k,n\in\IN$ the group $C_{2^k}^{2n}$ and $C_{2^k}^{2n+1}$ have the difference characteristics
$$\eth[C_{2^k}^{2n}]< \eth[C_4^n]-\frac1{2^{kn}}+\frac1{\sqrt{2^n}}\eth[C_{2^{k-1}}^n\times C_{2^k}^n]<1+\frac1{\sqrt{2^n}}\Big(\frac3{\sqrt{2}}+\frac4{\sqrt{3}}\Big)$$and
$$\eth[C_{2^k}^{2n+1}]<\eth[C_4^n\times C_{2^k}]+\frac1{\sqrt{2^n}}\eth[C_{2^{k-1}}^n\times C_{2^k}^{n+1}]<
\eth[C_{2k}]\cdot\Big(1+\frac3{\sqrt{2^{n+1}}}\Big)+
\frac4{\sqrt{3\cdot 2^n}}.
$$
\end{corollary}

\section{The difference characteristic of the groups $R\times U(R)$}

In this section we obtain an upper bound for the difference characteristics of
the groups $R\times U(R)$, which are products of the additive group of a ring
$R$ and the multiplicative group $U(R)$ of its units.


\begin{theorem}\label{t:RUR} For any finite ring $R$ and its multiplicative group $U(R)$ of units the set $B=\{(x,x):x\in U(R)\}$ is a difference base for the set $A=\{(x,y)\in R\times U(R):\exists z\in U(R)\;\; (y-1)z=x\}$ in the group $R\times U(R)$.
If the ring $R$ is local with maximal ideal $I_{\mathfrak  m}$ and the residue field $F=R/I_{\mathfrak m}$, then
$$\Delta[R\times U(R)]=|U(R)|-2+\Delta[I_{\mathfrak m}\times U(R)]+\Delta[R\times(1+I_{\mathfrak m})]\le |U(R)|-2+\frac4{\sqrt{3}}\frac{|R|}{|F|}\big(\sqrt{|F|-1}+\sqrt{|F|}).$$
and
$$\eth[R\times U(R)]<
\sqrt{1-\frac1{|F|}}+\frac{4}{\sqrt{3}}\Big(\frac1{\sqrt{|F|}}+\frac1{\sqrt{|F|-1}}\Big).
$$
\end{theorem}

\begin{proof}  Given a pair $(x,y)\in A$, we should find two elements $a,b\in U(R)$ such that $(a-b,ab^{-1})=(x,y)$. Since $(x,y)\in A$, there exists $z\in U(R)$ such that $(y-1)z=x$. Then for the pair $(a,b):=(yz,z)$ we get the required equality $(a-b,ab^{-1})=(yz-z,yzz^{-1})=((y-1)z,y)=(x,y)$.

Now assume that the ring $R$ is local and consider its (unique) maximal ideal $I_{\mathfrak m}$. Let $\pi:R\to R/I_{\mathfrak m}$ be the homomorphism of $R$ onto its residue field $F:=R/I_{\mathfrak m}$. It follows that $|R|=|F|\cdot|I_{\mathfrak m}|$. The maximality of the ideal $I_{\mathfrak m}$ guarantees that $U(R)=R\setminus I_{\mathfrak m}$ and $1+I_{\mathfrak m}=\pi^{-1}(\pi(1))$ is a multiplicative subgroup of $U(R)$. We claim that $(I_{\mathfrak m}\times U(R))\cup(R\times (1+I_{\mathfrak m}))\cup A=R\times U(R)$. Indeed, if a pair $(x,y)\in R\times U(R)$ does not belong to $(I_{\mathfrak m}\times U(R))\cup(R\times(1+I_{\mathfrak m}))$, then $x\in U(R)$ and $y\notin 1+I_{\mathfrak m}$. It follows that $1-y\notin I_{\mathfrak m}$ and hence the element $1-y$ is invertible, so we can find $z=(1-y)^{-1}x\in U(R)$ and conclude that $(x,y)\in A$. Theorem~\ref{t:KL} and the subadditivity of the difference size proved in Proposition~\ref{p:BGN}(3) guarantee that
$$
\begin{aligned}
\Delta&[R\times U(R)]\le \Delta[A]+\Delta[I_{\mathfrak m}\times U(R)]+\Delta[R\times (1+I_{\mathfrak m})]-2\le\\
&\le |U(R)|-2+\frac4{\sqrt{3}}(\sqrt{|I_{\mathfrak m}|\times |U(R)|}+\sqrt{|R|\times|I_{\mathfrak m}|})=
|U(R)|-2+\frac4{\sqrt{3}}\sqrt{|I_{\mathfrak m}|}
(\sqrt{|R|-|I_{\mathfrak m}|}+\sqrt{|R|})=\\
&=|U(R)|-2+\frac4{\sqrt{3}}|I_{\mathfrak m}|
\big(\sqrt{|R/I_{\mathfrak m}|-1}+\sqrt{|R/I_{\mathfrak m}|}\big)\le|U(R)|-2+\frac4{\sqrt{3}}\frac{|R|}{|F|}
\big(\sqrt{|F|-1}+\sqrt{|F|}\big).
\end{aligned}
$$
Dividing $\Delta[R\times U(R)]$ by $\sqrt{|R\times U(R)|}=\sqrt{|R|(|R|-|I_{\mathfrak m}|)}=|R|\sqrt{1-\frac1{|F|}}$, we get the required upper bound for the  difference characteristic
$$
\eth[R\times U(R)]<\sqrt{\frac{|U(R)|}{|R|}}+\frac4{\sqrt{3}}
\frac1{\sqrt{|F|^2-|F|}}(\sqrt{|F|-1}+\sqrt{|F|})=
\sqrt{1-\frac1{|F|}}+\frac4{\sqrt{3}}\Big(\frac1{\sqrt{|F|}}+\frac1{\sqrt{|F|-1}}\Big).
$$
\end{proof}

Combining Theorem~\ref{t:RUR} with Theorem~\ref{t:R*} describing the structure of the multiplicative groups of the Galois rings $GR(p^k,r)$, we get the following two corollaries.

\begin{corollary} Let $p$ be a prime number and $k,r$ be natural numbers such that either $p\ge 3$ or $p=2$ and $k\le 2$. The group $G=C_{p^k}^r\times C_{p^{k-1}}^r\times C_{p^r-1}$ has difference characteristic
$$\eth[G]<
\sqrt{1-\frac{1}{p^r}}+\frac4{\sqrt{3}}\Big(\frac1{\sqrt{p^r}}+\frac1{\sqrt{p^r-1}}\Big)=1+O\big(\tfrac1{p^{r/2}}\big).$$
\end{corollary}

\begin{corollary} For any natural numbers $r$ and $k\ge 3$ the group
$$G:=C_{2^k}^r\times C_{2^{k-1}}^{r-1}\times C_{2^{k-2}}\times C_2\times C_{2^r-1}$$ has difference characteristic
$$\eth[G]<
\sqrt{1-\frac{1}{2^r}}+\frac4{\sqrt{3}}\Big(\frac1{\sqrt{2^r}}+\frac1{\sqrt{2^r-1}}\Big)=
1+O\big(\tfrac1{2^{r/2}}\big)
.$$
\end{corollary}


\section{The results of computer calculations}

In Table~~\ref{tab:abel} we present the results of computer calculations of the difference
sizes of all non-cyclic Abelian groups $G$ of order $12\le|G|<96$. In this
table
$$lb[G]:=\left\lceil\frac{1+\sqrt{4|G|+4|G_2|-3}}2\,\right\rceil$$is
the lower bound given in Corollary~\ref{c:lower}.

\begin{table}[ht]
\caption{Difference sizes of non-cyclic Abelian groups $G$ of order
$12\le|G|<96$}\label{tab:abel} {\small
\begin{tabular}{|c|c|c|c|c|c|c|c|c|c|}
\hline
$G\phantom{\big|^|\!\!\!}$& $(C_2)^2\times C_3$&$C_2\times C_8$
&$(C_4)^2$&$(C_2)^2\times C_4$&$(C_2)^4$& $C_2\times (C_3)^2$&
$(C_2)^2\times C_{5}$\\
\hline
$lb[G]$  & 5 &5&5 &6&6& 5 &6\\
$\Delta[G]$  & 5&5 &6 &6&6 & 5&6\\
$\eth[G]$& 1,4433... & 1,25 & 1,5 & 1,5 & 1,5 & 1,1785... & 1,3416...\\
\hline
\hline
$G\phantom{\big|^|\!\!\!}$  & $C_2{\times}C_3{\times}C_4$&
$(C_2)^3\times C_3$ &$(C_5)^2$ &$C_3\times C_9$&$(C_3)^3$&
$(C_2)^2\times C_{7}$&$C_2\times C_{16}$  \\
\hline
$lb[G]$   & 6 & 6 & 6 &6&6 &6&7\\
$\Delta[G]$   & 6 & 6 &  6&6 &6&6&7\\
$\eth[G]$ & 1,2247... & 1,2247... & 1,2 & 1,1547... & 1,1547...&
1,1338... & 1,2374...\\
\hline
\hline
$G\phantom{\big|^|\!\!\!}$&  $C_4\times C_8$ & $(C_2)^2\times C_8$ &
$C_2\times (C_{4})^2$ & $(C_2)^3\times C_{4}$&$(C_2)^5$&
$(C_6)^2$&$(C_2)^2\times C_{9}$ \\
\hline
$lb[G]$  & 7 & 7 & 7 & 8 &9 &7&7\\
$\Delta[G]$  & 7 & 7 & 8 & 8 &10&7&7 \\
$\eth[G]$ & 1,2374... & 1,2374... & 1,4142... & 1,4142... & 1,7677... &
1,1666...& 1,1666...\\
\hline
\hline
$G\phantom{\big|^|\!\!\!}$ & $(C_3)^2\times C_{4}$ & $(C_2)^3\times
C_{5}$&$C_2{\times}C_4{\times}C_5$& $(C_2)^2\times C_{11}$ &
$(C_3)^2\times C_{5}$&  $\!\!C_2{\times}C_3{\times}C_8\!\!\!$&
$C_3\times(C_4)^2$\\
\hline
$lb[G]$ & 7 & 8& 7 & 8 & 8& 8&8\\
$\Delta[G]$ &  7 & 8& 8 & 8 & 8& 8&8\\
$\eth[G]$ & 1,1666... & 1,2649... & 1,2649... & 1,2060... & 1,1925... &
1,1547... & 1,1547...\\
\hline
\hline
$G\phantom{\big|^|\!\!\!}$&
$\!\!(C_2)^2{\times}C_3{\times}C_4\!\!$&$(C_2)^4\times C_3$&
$(C_7)^2$&$C_2\times (C_5)^2$& $(C_2)^2\times C_{13}$& $C_6\times
C_{9}$& $C_2\times (C_3)^3$\\
\hline
$lb[G]$   & 8 &9& 8&8 & 8& 8  & 8\\
$\Delta[G]$   & 9 &10& 9&8 & 9& 9 & 9 \\
$\eth[G]$ & 1,2990... & 1,4433... & 1,2857... & 1,1313...& 1,2480... &
1,2247... & 1,2247...\\
\hline
\hline
$G\phantom{\big|^|\!\!\!}$ & $C_2{\times}C_4{\times}C_7$ &
$(C_2)^3\times C_{7}$ & $\!\!(C_2)^2{\times}C_3{\times}C_{5}\!\!\!$ &
$(C_3)^2\times C_{7}$ & $C_2\times C_{32}$ & $C_4\times C_{16}$&
$\!\!C_2{\times}C_4{\times}C_8\!\!\!$ \\
\hline
$lb[G]$  & 9 & 9 & 9 & 9 & 9 & 9& 9 \\
$\Delta[G]$  & 9 & 10 & 9 & 9 & 10 & 10 & 10 \\
$\eth[G]$ & 1,2026... & 1,3363... & 1,1618... & 1,1338... & 1,25 & 1,25
& 1,25\\
\hline
\hline
$G\phantom{\big|^|\!\!\!}$& $(C_2)^2{\times} C_{16}$ & $(C_8)^2$ &
$(C_4)^3$ & $(C_2)^3\times C_{8}$ &$(C_2)^2{\times} (C_{4})^2\!\!$&
$(C_2)^4\times C_{4}$ &$(C_2)^6$ \\
\hline
$lb[G]$  & 9 & 9 & 9& 10 & 10 &11&12\\
$\Delta[G]$  & 10 & 10& 11&  11 & 12&12&14 \\
$\eth[G]$ & 1,25 & 1,25 & 1,375 & 1,375 & 1,5 & 1,5 & 1,75\\
\hline
\hline
$G\phantom{\big|^|\!\!\!}$& $(C_2)^2{\times} C_{17}$ &
$C_2{\times}C_4{\times}C_9$& $(C_3)^2\times C_{8}$&
$\!\!C_2{\times}(C_3)^2{\times}C_4\!\!\!$ & $(C_2)^3{\times}(C_3)^2\!\!$
& $(C_2)^3\times C_{9}$& $C_3\times (C_{5})^2$ \\
\hline
$lb[G]$& 9 & 10 & 9  & 10 & 10 & 10 & 10\\
$\Delta[G]$& 10 & 10 & 10 & 10 & 11 & 11 & 10 \\
$\eth[G]$ & 1,2126... & 1,1785... & 1,1785... & 1,1785... & 1,2963... &
1,2963... & 1,1547...\\
\hline
\hline
$G\phantom{\big|^|\!\!\!}$& $(C_2)^2{\times} C_{19}$ &
$C_2{\times}C_8{\times}C_5$&$(C_4)^2\times C_{5}$&
$\!(C_2)^2{\times}C_4{\times}C_5\!\!$& $(C_2)^4\times C_{5}$& $(C_9)^2$&
$(C_3)^4$\\
\hline
$lb[G]$ & 10 & 10 & 10 & 10 &11&10 &10\\
$\Delta[G]$ & 11 & 11 &  11 & 12 &12&11 &12\\
$\eth[G]$ & 1,2617... & 1,2298... & 1,2298...& 1,3416... & 1,3416... &
1,2222... & 1,3333...\\
\hline
\hline
$G\phantom{\big|^|\!\!\!}$& $(C_3)^2\times C_9$ & $C_3\times C_{27}$ &
$\!(C_2)^2{\times}C_3{\times}C_7\!\!\!$ & $(C_2)^3\times C_{11}$ &
$C_2{\times}C_4{\times}C_{11}\!\!$ &
$\!C_2{\times}(C_3)^2{\times}C_{5}\!\!\!$ & $(C_2)^2{\times}C_{23}$\\
\hline
$lb[G]$&  10 & 10 & 10 & 11 & 10 & 10 & 11 \\
$\Delta[G]$&  11 & 11 & 11 & 12 & 12 & 11 & 12 \\
$\eth[G]$& 1,2222... & 1,2222... & 1,2001... & 1,2792... & 1,2792... &
1,1595... & 1,2510...\\
\hline
\end{tabular}
}
\end{table}

\section{Acknowledgment}

The authors would like to express their sincere thanks to Oleg Verbitsky who
turned their attention to the theory of difference sets and their relation with
difference bases, to Orest Artemovych for consultations on Galois rings, to
Alex Ravsky for valuable discussions on difference sizes of product groups, and
to MathOverflow users Seva and Jeremy Rickard for valuable comments.


\begin{thebibliography}{9}


\bibitem{BGN} T.~Banakh, V.~Gavrylkiv, O.~Nykyforchyn, {\em Algebra in
superextension of groups, I:
zeros and commutativity}, Algebra Discr. Math. {\bf 3} (2008), 1--29.

\bibitem{BG} T.~Banakh, V.~Gavrylkiv,  {\em Difference bases in cyclic groups},
    preprint (https://arxiv.org/abs/1702.02631). 



\bibitem{BF} G.~Bini, F.~Flamini, {\em Finite commutative rings and their applications}, Kluwer Academic Publishers, Boston, MA, 2002.

\bibitem{Bose} R.C. Bose, {\em An affine analogue of Singer's theorem}, Journal
    of the Indian Mathematical Society {\bf 6} (1942), 1--15.

\bibitem{Chowla} R.C.~Bose, S.~Chowla, {\em Theorems in the additive theory of
    numbers}, Comment. Math. Helvetici {\bf 37} (1962-63) 141--147.






\bibitem{Golay} M.~Golay, {\em Notes on the representation of $1,\,2,\,\ldots
    ,\,n$ by differences}, J. London Math. Soc. (2) {\bf 4} (1972) 729--734.




\bibitem{Leech} J.~Leech, {\em On the representation of $1,2,\dots,n$ by
differences}, J. London Math. Soc. {\bf31} (1956), 160--169.

\bibitem{KL} G.~Kozma, A.~Lev, {\em Bases and decomposition numbers of
finite groups}, Arch. Math. (Basel) {\bf 58}:5 (1992), 417--424.


\bibitem{PSZ} A.~Pott, Kai-Uwe Schmidt, Yue Zhou, {\em Semifields, relative difference sets, and bent functions. Algebraic curves and finite fields}, Radon Ser. Comput. Appl. Math., {\bf16}, De Gruyter, Berlin, (2014) 161--178.

\bibitem{RR} L.~R\'edei, A.~R\'enyi, {\em On the representation of the
numbers $1,2,\dots, N$ by means of differences}, Mat. Sbornik N.S.
{\bf 24}(66) (1949), 385--389.

\bibitem{Rob} D.~Robinson, {\em A course in the theory of groups},
Springer-Verlag, New York, 1996.


\bibitem{Rusza} I.Z.~Ruzsa, {\em Solving a linear equation in a set of integers
    I}, Acta Arithmetica LXV. {\bf 3} (1993) 259--282.


\bibitem{Singer} J.~Singer, {\em A theorem in finite projective geometry
and some applications to number theory}, Trans. Amer. Math. Soc.
{\bf 43}:3 (1938), 377--385.



\end{thebibliography}
\end{document}